\documentclass[10pt,a4paper,twoside,fleqn]{amsart}
\usepackage{amssymb,amsmath,amsthm,latexsym}

\usepackage[varg]{pxfonts}

\newtheorem{theorem}{Theorem}[section]

\newtheorem{corollary}[theorem]{Corollary}

\newtheorem{lemma}[theorem]{Lemma}

\newtheorem{proposition}[theorem]{Proposition}

\newcommand{\ie}{\emph{i.e.\ }} 
\newcommand{\f}{\varphi}
\newcommand{\tg}{\tilde{g}}
\newcommand{\n}{\nabla}
\newcommand{\nn}{\tilde{\n}}
\newcommand{\M}{(M,\f,\xi,\eta,g)}
\newcommand{\G}{\mathcal{G}}
\newcommand{\I}{\mathcal{I}}
\newcommand{\W}{\mathcal{W}}
\newcommand{\R}{\mathbb R}

\newcommand{\X}{\mathfrak X}
\newcommand{\F}{\mathcal{F}}

\newcommand{\ta}{\theta}
\newcommand{\om}{\omega}
\newcommand{\lm}{\lambda}

\newcommand{\im}{\operatorname{im}}
\newcommand{\Span}{\operatorname{span}}
\newcommand{\Id}{\operatorname{Id}}

\newcommand{\thmref}[1]{Theorem~\ref{#1}}
\newcommand{\lemref}[1]{Lemma~\ref{#1}}
\newcommand{\coref}[1]{Corollary~\ref{#1}}
\newcommand{\propref}[1]{Proposition~\ref{#1}}

\begin{document}

\title{On the Structure Tensors of Almost Contact B-Metric Manifolds}

\author{Hristo Manev}
\address{Medical University of Plovdiv, Faculty of Pharmacy,
Department of Pharmaceutical Sciences, 15-A Vasil Aprilov Blvd.
} \email{hmanev@uni-plovdiv.bg}
\newcommand{\AuthorNames}{H. Manev}



\thanks{This paper is partially supported by project NI13-FMI-002
of the Scientific Research Fund, Plovdiv University, Bulgaria 
}

\begin{abstract}
The space of the structure (0,3)-tensors of the covariant
derivatives of the structure endomorphism and the metric on almost
contact B-metric manifolds is considered. A known decomposition of
this space in orthogonal and invariant subspaces with respect to
the action of the structure group is used. We determine the
corresponding components of the structure tensor and consider the
case of the lowest dimension 3 of the studied manifolds. Some
examples are commented.
\end{abstract}


\maketitle



\section*{Introduction}\label{sec-intro}

The idea of decomposition of the space of the structure
(0,3)-tensors, generated by the covariant derivative of the
fundamental tensor of type $(1,1)$, is used by different authors
in order to  obtain classifications of manifolds with additional
tensor structures. For example, let us mention the classification
of almost Hermitian manifolds given in \cite{GrHe}, of almost
complex manifolds with Norden metric -- in \cite{GaBo}, of almost
contact metric manifolds -- in \cite{AlGa}, of almost contact
manifolds with B-metric -- in \cite{GaMiGr}, of Riemannian almost
product manifolds -- in \cite{Nav}, of Riemannian manifolds with
traceless almost product structure -- in \cite{StaGri}, of almost
paracontact metric manifolds -- in \cite{NakZam}, of almost
paracontact Riemannian manifolds of type $(n,n)$ -- in
\cite{ManSta01}.

The almost contact structure is well studied, especially in the
case when it is equipped with a compatible Riemannian (or
pseudo-Riemannian) metric, \ie the almost contact endomorphism
$\f$ acts as an isometry with respect to the metric in each
tangent fibre of the contact distribution $H=\ker(\eta)$, where
$\eta$ is the contact 1-form. An indefinite counterpart is the
almost contact B-metric structure, \ie $\f$ acts as an
anti-isometry with respect to B-metric in $H$.

The goal of this work is the description of the structure tensor
generated by the covariant derivative of $\f$ and the B-metric by
its components in the different basic classes of the
classification of the almost contact B-metric manifolds made by G.
Ganchev, V. Mihova, K. Gribachev in \cite{GaMiGr}. The case of the
lowest dimension 3 of the studied manifolds is considered and it
is establish that four of the basic classes are restricted to the
spacial class with zero structure tensors.

The paper is organized as follows. In Sect.~\ref{sec-mfds} we
recall some facts about the almost contact manifolds with
B-metric. 
In Sect.~\ref{sec-dec} we decompose the vector space of the
structure tensors on the considered manifolds.
In Sect.~\ref{sec-com3} we deduce the components of the structure
tensor in the case of the lowest dimension 3.
In Sect.~\ref{sec-exm} we comment some examples in relation with
the above investigations.




\section{Almost contact manifolds with B-metric}\label{sec-mfds}

Let $(M,\f,\xi,\eta,g)$ be an almost contact manifold with
B-met\-ric or an \emph{almost contact B-metric manifold}, \ie $M$
is a $(2n+1)$-dimensional differentiable manifold with an almost
contact structure $(\f,\xi,\eta)$ consisting of an endomorphism
$\f$ of the tangent bundle, a vector field $\xi$, its dual 1-form
$\eta$ as well as $M$ is equipped with a pseu\-do-Rie\-mannian
metric $g$  of signature $(n+1,n)$, such that the following
algebraic relations are satisfied: \cite{GaMiGr}
\begin{equation}\label{str}
\begin{array}{c}
\f\xi = 0,\quad \f^2 = -\Id + \eta \otimes \xi,\quad
\eta\circ\f=0,\quad \eta(\xi)=1,\quad 
g(\f x, \f y) = - g(x,y) + \eta(x)\eta(y)
\end{array}
\end{equation}
for any $x$, $y$ of the algebra $\X(M)$ of the smooth vector
fields on $M$. Further $x$, $y$, $z$ will stand for arbitrary
elements of $\X(M)$ or vectors in the tangent space $T_pM$ at
$p\in M$.

The associated metric $\tg$ of $g$ on $M$ is defined by
\(\tg(x,y)=g(x,\f y)+\eta(x)\eta(y)\).  The manifold
$(M,\f,\xi,\eta,\tg)$ is also an almost contact B-metric manifold.
Both metrics $g$ and $\tg$ are necessarily of signature $(n+1,n)$.
The Levi-Civita connection of $g$ and $\tg$ will be denoted by
$\n$ and $\nn$, respectively.

Let us denote the structure group of $\M$ by $\G$. It is
determined by $\G=\mathcal{O}(n;\mathbb{C})\times\I$, where $\I$
is the identity on $\Span(\xi)$ and
$\mathcal{O}(n;\mathbb{C})=\mathcal{GL}(n;\mathbb{C})\cap
\mathcal{O}(n,n)$, \ie $\G$ consists of the real square matrices
of order $2n+1$ of the following type
\[
\left(%
\begin{array}{r|c|c}
  A & B & \vartheta^\top\\ \hline
  -B & A & \vartheta^\top\\ \hline
  \vartheta & \vartheta & 1 \\
\end{array}%
\right),\qquad %
\begin{array}{l}
  A^\top A-B^\top B=I_n,\\%
  B^\top A+A^\top B=O_n,
\end{array}%
\quad A, B\in \mathcal{GL}(n;\mathbb{R}),
\]
where $\vartheta$ and its transpose $\vartheta^\top$ are the zero row
$n$-vector and the zero column $n$-vector; $I_n$ and $O_n$ are the
unit matrix and the zero matrix of size $n$, respectively.

A classification of the almost contact B-metric manifolds is given
in \cite{GaMiGr}. This classification, consisting of eleven basic
classes $\F_1$, $\F_2$, $\dots$, $\F_{11}$,  is made with respect
to the tensor  $F$ of type (0,3) defined by
\begin{equation}\label{F=nfi}
F(x,y,z)=g\bigl( \left( \nabla_x \f \right)y,z\bigr),
\end{equation}
which have the following properties
\begin{equation}\label{F-prop}
F(x,y,z)=F(x,z,y)=F(x,\f y,\f z)+\eta(y)F(x,\xi,z)
+\eta(z)F(x,y,\xi).
\end{equation}

The intersection of the basic classes is the special class $\F_0$
determined by the condition $F(x,y,z)=0$. Hence $\F_0$ is the
class of the almost contact B-metric manifolds with $\n$-parallel
structures, \ie $\n\f=\n\xi=\n\eta=\n g=\n \tg=0$.

If $\left\{e_i;\xi\right\}$ $(i=1,2,\dots,2n)$ is a basis of
$T_pM$ and $\left(g^{ij}\right)$ is the inverse matrix of the
matrix $\left(g_{ij}\right)$ of $g$, then the following 1-forms
are associated with $F$:
\begin{equation}\label{t}
\theta(z)=g^{ij}F(e_i,e_j,z),\quad
\theta^*(z)=g^{ij}F(e_i,\f e_j,z), \quad \omega(z)=F(\xi,\xi,z).
\end{equation}
These 1-forms are known also as the Lee forms. Obviously, the identities
$\om(\xi)=0$ and $\ta^*\circ\f=-\ta\circ\f^2$ are always valid.

Further we use the following characteristic conditions of the
basic classes: \cite{Man8}
\begin{equation}\label{Fi}
\begin{array}{rl}
\F_{1}: &F(x,y,z)=\frac{1}{2n}\bigl\{g(x,\f y)\ta(\f z)+g(\f x,\f
y)\ta(\f^2 z)+g(x,\f z)\ta(\f y)+g(\f x,\f z)\ta(\f^2 y)
\bigr\};\\[4pt]
\F_{2}: &F(\xi,y,z)=F(x,\xi,z)=0,\quad
              F(x,y,\f z)+F(y,z,\f x)+F(z,x,\f y)=0,\quad \ta=0;\\[4pt]
\F_{3}: &F(\xi,y,z)=F(x,\xi,z)=0,\quad
              F(x,y,z)+F(y,z,x)+F(z,x,y)=0;\\[4pt]
\F_{4}: &F(x,y,z)=-\frac{1}{2n}\ta(\xi)\bigl\{g(\f x,\f y)\eta(z)+g(\f x,\f z)\eta(y)\bigr\};\\[4pt]
\F_{5}: &F(x,y,z)=-\frac{1}{2n}\ta^*(\xi)\bigl\{g( x,\f y)\eta(z)+g(x,\f z)\eta(y)\bigr\};\\[4pt]
\F_{6}: &F(x,y,z)=F(x,y,\xi)\eta(z)+F(x,z,\xi)\eta(y),\quad \\[4pt]
                &F(x,y,\xi)=F(y,x,\xi)=-F(\f x,\f y,\xi),\quad \ta=\ta^*=0; \\[4pt]
\F_{7}: &F(x,y,z)=F(x,y,\xi)\eta(z)+F(x,z,\xi)\eta(y),\quad 
                F(x,y,\xi)=-F(y,x,\xi)=-F(\f x,\f y,\xi); \\[4pt]
\F_{8}: &F(x,y,z)=F(x,y,\xi)\eta(z)+F(x,z,\xi)\eta(y),\quad 
                F(x,y,\xi)= F(y,x,\xi)=F(\f x,\f y,\xi); \\[4pt]
\F_{9}: &F(x,y,z)=F(x,y,\xi)\eta(z)+F(x,z,\xi)\eta(y),\quad 
                F(x,y,\xi)=-F(y,x,\xi)=F(\f x,\f y,\xi); \\[4pt]
\F_{10}: &F(x,y,z)=F(\xi,\f y,\f z)\eta(x); \\[4pt]
\F_{11}:
&F(x,y,z)=\eta(x)\left\{\eta(y)\om(z)+\eta(z)\om(y)\right\}.
\end{array}
\end{equation}

\section{A decomposition of the space of the structure tensors}\label{sec-dec}

Let us consider $T_p{M}$ at arbitrary $p\in M$ as a
$(2n+1)$-dimensional vector space equipped with almost contact
B-metric structure $(\f,\xi,\eta,g)$. Let $\F$ be the vector space
of all tensors $F$ of type (0,3) over $T_p{M}$ having properties
\eqref{F-prop}, \ie
\begin{equation}\label{Fspace}
\begin{split}
\mathcal{F} =\bigl\{F(x,y,z)\in \R \;|\; F(x,y,z)&=F(x,z,y)=F(x,\f
y,\f z)+\eta(y)F(x,\xi,z)+\eta(z)F(x,y,\xi)\bigr\}.
\end{split}
\end{equation}

The metric $g$ induces an inner product
$\langle\cdot,\cdot\rangle$ on $\F$ defined by
\[
\langle F',F''\rangle=g^{iq}g^{jr}g^{ks}\allowbreak{}
F'(e_i,e_j,e_k)\allowbreak{}F''(e_q,e_r,e_s)
\]
for any $F',F''\in\F$ and a basis $\left\{e_i\right\}$
$(i=1,2,\dots,2n+1)$ of $T_p{M}$.

The standard representation of the structure group $\G$
 in $T_p{M}$ induces a natural representation $\lm$ of $\G$ in
 $\F$ as follows
\begin{equation}\label{la}
\left((\lm a)F\right)(x,y,z)=F\left(a^{-1}x,a^{-1}y,a^{-1}z\right)
\end{equation}
 for any
$a\in \G$ and $F\in\F$, so that for $F',F''\in \F$
\[
\langle(\lm a)F',(\lm a)F''\rangle=\langle F',F''\rangle.
\]

The decomposition $x=-\f^2 x+\eta (x)\xi$ generates the projectors
$h$ and $v$ on $T_p{M}$ determined by $h(x)=-\f^2 x$ and
$v(x)=\eta (x)\xi$, having the properties $h \circ h = h$, $v
\circ v = v$, $h \circ v = v \circ h = 0$.

Therefore, we have the orthogonal decomposition $T_p{M}=h(T_p{M})
\oplus v(T_p{M})$.

Bearing in mind these projectors on $T_p{M}$, we construct a
partial decomposition of $\F$ as follows.

At first, we define the operator $p_1:\F\rightarrow\F$ by
\begin{equation}\label{p1}
p_1(F)(x,y,z)=-F(\f^2x,\f^2y,\f^2z).
\end{equation}
It is easy to check the following
\begin{lemma}\label{lem-p1}
The operator $p_1$ has the following properties:
\begin{itemize}
    \item[(i)]
    $p_1\circ p_1 = p_1$;
    \item[(ii)]
    $\langle p_1(F'),F''\rangle=\langle F',p_1(F'')\rangle,\quad F', F'' \in\F$;
    \item[(iii)]
    $p_1\circ (\lm a)=(\lm a)\circ p_1$.
\end{itemize}
\end{lemma}

According to \lemref{lem-p1} we have the following orthogonal
decomposition of $\F$ by the image and the kernel of $p_1$:
\begin{equation}
\begin{array}{l}
\W_1=\im(p_1)=\{F\in\F\;|\; p_1(F)=F\}, \quad 
\W_1^\bot=\ker(p_1)=\{F\in\F\;|\; p_1(F)=0\}.
\end{array}
\end{equation}

Further, we consider the operator $p_2:
\W_1^\bot\rightarrow\W_1^\bot$ defined by
\begin{equation}\label{p2}
p_2(F)(x,y,z)
=\eta(y)F(\f^2{x},\xi,\f^2{z})
+\eta(z)F(\f^2{x},\f^2{y},\xi)
\end{equation}
for which we obtain
\begin{lemma}\label{lem-p2}
The operator $p_2$ has the following properties:
\begin{itemize}
    \item[(i)]
    $p_2\circ p_2 = p_2$;
    \item[(ii)]
    $\langle p_2(F'),F''\rangle=\langle F',p_2(F'')\rangle,\quad F', F'' \in\W_1^\bot$;
    \item[(iii)]
    $p_2\circ (\lm a)=(\lm a)\circ p_2$.
\end{itemize}
\end{lemma}
Then, bearing in mind  \lemref{lem-p2}, we obtain
\begin{equation}
\begin{array}{l}
\W_2=\im(p_2)=\left\{F\in\W_1^\bot\;|\; p_2(F)=F\right\}, \quad 
\W_2^\bot=\ker(p_2)=\left\{F\in\W_1^\bot\;|\; p_2(F)=0\right\}.
\end{array}
\end{equation}

Finally, we consider the operator $p_3:
\W_2^\bot\rightarrow\W_2^\bot$ defined by
\[
p_3(F)(x,y,z)
=\eta(x)F(\xi,\f^2{y},\f^2{z})
\]
and we get the following
\begin{lemma}\label{lem-p3}
The operator $p_3$ has the properties:
\begin{itemize}
    \item[(i)]
    $p_3\circ p_3 = p_3$;
    \item[(ii)]
    $\langle p_3(F'),F''\rangle=\langle F',p_3(F'')\rangle,\quad F', F'' \in\W_2^\bot$;
    \item[(iii)]
    $p_3\circ (\lm a)=(\lm a)\circ p_3$.
\end{itemize}
\end{lemma}
By virtue of \lemref{lem-p3} we have
\begin{equation}
\begin{array}{l}
\W_3=\im(p_3)=\{F\in\W_2^\bot\;|\; p_3(F)=F\}, \quad 
\W_4=\ker(p_3)=\{F\in\W_2^\bot\;|\; p_3(F)=0\}.
\end{array}
\end{equation}


From \lemref{lem-p1}, \lemref{lem-p2} and \lemref{lem-p3} we have
immediately
\begin{theorem}\label{thm-W1234}
The decomposition $\F=\W_1\oplus\W_2\oplus\W_3\oplus\W_4 $ is
orthogonal and invariant under the action of $\G$. The subspaces
$\W_i$ $(i=1,2,3,4)$ are determined by
\begin{equation}\label{WF}
\begin{array}{l}
\W_1: \; F(x,y,z) = - F(\f^2 x, \f^2 y, \f^2z), \\[4pt]
\W_2: \; F(x,y,z) = \eta (y) F(\f^2 x, \xi, \f^2 z)+\eta (z) F(\f^2 x, \f^2 y, \xi),\\[4pt]
\W_3: \; F(x,y,z) = \eta (x) F(\xi, \f^2 y, \f^2 z),\\[4pt]
\W_4: \; F(x,y,z) = -\eta (x) \left\{\eta(y) F(\xi, \xi,\f^2 z) +
\eta(z) F(\xi, \f^2 y, \xi)\right\}.
\end{array}
\end{equation}
\end{theorem}

Obviously, we have
\begin{equation}\label{p1234}
F=p_1(F)+p_2(F)+p_3(F)+p_4(F).
\end{equation}

\begin{corollary}\label{cor-W1234}
The subspaces $\W_i\ (i = 1, 2, 3, 4)$ are characterized as
follows:
\begin{equation}\label{WFhv}
\begin{array}{l}
\W_1 = \bigl\{F \in \F \;|\; F(v(x),y,z) = F(x, v(y), z) = F(x, y, v(z)) = 0\bigr\}, \\[4pt]
\W_2 = \bigl\{F \in \F \;|\; F(v(x),y,z) = F(x, h(y), h(z)) = 0\bigr\}, \\[4pt]
\W_3 = \bigl\{F \in \F \;|\; F(h(x),y,z) = F(x, v(y), z) = F(x, y, v(z)) = 0\bigr\}, \\[4pt]
\W_4 = \bigl\{F \in \F \;|\; F(h(x),y, z) = F(x, h(y), h(z)) = 0\bigr\}.
\end{array}
\end{equation}
\end{corollary}

According to \eqref{WF}, \eqref{WFhv}  and \eqref{t}
we obtain the following
\begin{corollary}\label{cor-t-Wi}
The Lee forms of $F$ have the following properties in each of the
sub\-spaces $\W_{i}$ $(i=1,2,3,4)$$:$\\[4pt]
\begin{tabular}{rl}
    $($i$)$& If $F\in\W_{1}$, then $\ta\circ v=\ta^*\circ
    v=\om=0;$ \\[4pt]
    $($ii$)$& If $F\in\W_{2}$, then $\ta\circ h=\ta^*\circ
    h=\om=0;$\\[4pt]
    $($iii$)$& If $F\in\W_{3}$, then $\ta=\ta^*=\om=0;$ \\[4pt]
    $($iv$)$& If $F\in\W_{4}$, then $\ta=\ta^*=0$.
\end{tabular}
\end{corollary}

Further we continue the decomposition of the subspaces $\W_i$ $(i = 1, 2, 3, 4)$ of $\F$.

\subsection{The decomposition of $\W_1$}

Let us consider the $2n$-dimensional distribution $H=\ker(\eta)$ of the tangent bundle of $\M$,
the endomorphism $J=\f|_H$ and the metric $h=g|_H$, where $\f|_H$ and $g|_H$ are the restrictions
of $\f$ and $g$ on $H$, respectively. Let us remark that $J$ and $h$ are an almost
complex structure and  a Norden metric on $H$, respectively, \ie
\begin{equation}\label{norden}
J^2=-\Id,\quad h(Jx,Jy)=-h(x,y) .
\end{equation}
Then $(H,J,h)$ can be considered as an almost complex manifold with Norden metric.

Moreover, the subspace $\W_1$ coincides with the restriction of $\F$ on $H$.
%
%
By this reason the decomposition of $\W_1$ is made as
the decomposition,  known from \cite{GaBo}, of the corresponding space of $\F$
for an almost complex manifold with Norden metric. Then we obtain the following
\begin{proposition}\label{prop-F123}
Let $F\in \F$ and $F_{i} \: (i=1,2,3)$ be the projections of $F$
on the subspaces $\F_{i}$, respectively. Then
\begin{equation}\label{F123}
\begin{split}
&F_{1}(x,y,z)=\frac{1}{2n}\bigl\{g(\f x,\f y)\ta(\f^2 z) +g(x,\f
y)\ta(\f z)+g(\f x,\f z)\ta(\f^2 y)+g(x,\f z)\ta(\f y)\bigr\}; \\[4pt]
&F_{2}(x,y,z)=-\frac{1}{4}\bigl\{ %
F(\f^2 x,\f^2 y,\f^2 z)+F(\f^2 y,\f^2 z,\f^2 x)-F(\f y,\f^2 z,\f x)+F(\f^2 x,\f^2 z,\f^2 y)  \\[4pt]
&\phantom{F_{2}(x,y,z)=} %
+F(\f^2 z,\f^2 y,\f^2 x)-F(\f z,\f^2 y,\f
x)\bigr\}-\frac{1}{2n}\bigl\{g(\f x,\f y)\ta(\f^2 z) +g(x,\f
y)\ta(\f z) \\[4pt]
&\phantom{F_{2}(x,y,z)=} +g(\f x,\f z)\ta(\f^2 y)+g(x,\f z)\ta(\f
y)\bigr\};  \\[4pt]
&F_{3}(x,y,z)=-\frac{1}{4}\bigl\{%
F(\f^2 x,\f^2 y,\f^2 z)-F(\f^2 y,\f^2 z,\f^2 x) +F(\f y,\f^2 z,\f
x)+F(\f^2 x,\f^2 z,\f^2 y)\\[4pt]
&\phantom{F_{3}(x,y,z)=}%
-F(\f^2 z,\f^2 y,\f^2 x) +F(\f z,\f^2 y,\f x)\bigr\}.
\end{split}
\end{equation}
\end{proposition}

Therefore, the component of $F$ on $\W_1$ is
\begin{equation}\label{p1F}
p_1(F)=F_1+F_2+F_3.
\end{equation}

\subsection{The decomposition of $\W_2$}

Let us consider linear operators $L_{j}:\ \W_2\rightarrow \W_2$
$(j=1,2)$ defined by
\begin{equation}\label{L12}
\begin{split}
&L_{1}(F)(x,y,z)=F(\f x,\f y,\xi)\eta(z)+F(\f x,\f z,\xi)\eta(y),  \\[4pt]
&L_{2}(F)(x,y,z)=F(\f^2 y,\f^2 x,\xi)\eta(z)+F(\f^2 z,\f^2
x,\xi)\eta(y).
\end{split}
\end{equation}

It is easy to check the following
\begin{lemma}\label{lem-L10}
The linear operator $L_{j}$ $(j=1,2)$ is an involutive isometry on
$\W_2$ and it is invariant with respect to the group $\G$, \ie
\[
\begin{array}{c}
L_{j}\circ L_{j}=\Id_{\W_2}, \qquad \langle
L_{j}(F'),L_{j}(F'')\rangle=\langle F',F'' \rangle, \qquad
L_{j}((\lm a)F)=(\lm a)(L_{j}(F)),
\end{array}
\]
where $F',\ F''\in\W_2$ and $\lm a$ is determined by \eqref{la}.
\end{lemma}

Therefore, $L_{1}$ has two eigenvalues $+1$ and $-1$, and the
corresponding eigen\-spaces
\[
\W_2^+=\left\{F\in \W_2\ \vert\ L_{1}(F)= F\right\},\qquad
\W_2^-=\left\{F\in \W_2\ \vert\ L_{1}(F)=- F\right\}
\]
are invariant orthogonal subspaces of $\W_2$.

In order to decompose $\W_2^+$ and $\W_2^-$, we use the
operator $L_{2}$ on $\W_2^+$ and $\W_2^-$,  respectively.
Let us denote the corresponding eigenspaces $\W_{2,k}$ $(k=1,2,3,4)$ by
\[
\begin{array}{ll}
\W_{2,1}=\left\{\W_2^+\ \vert\ L_{2}(F)=- F\right\},\qquad&
\W_{2,2}=\left\{\W_2^-\ \vert\ L_{2}(F)=- F\right\},\\[4pt]
\W_{2,3}=\left\{\W_2^+\ \vert\ L_{2}(F)= F\right\},\qquad&
\W_{2,4}=\left\{\W_2^-\ \vert\ L_{2}(F)= F\right\}.
\end{array}
\]

Thus we establish the truthfulness of the following
\begin{theorem}\label{thm-T1k}
The decomposition $\W_2=\W_{2,1}\oplus\W_{2,2}\oplus\W_{2,3}\oplus\W_{2,4} $ is
orthogonal and invariant with respect to the structure group.
\end{theorem}


\begin{proposition}\label{prop-F4-9}
Let $F\in\F$ and $F_{j}$ $(j=4,\dots,9)$ be the projections of $\F$
in the classes $\F_{j}$. Then we have
\begin{subequations}\label{F4-9}
\begin{equation}
\begin{split}
&F_{4}(x,y,z)=-\frac{\ta(\xi)}{2n}\bigl\{g(\f x,\f y)\eta(z)+g(\f x,\f z)\eta(y)\bigr\};  \\[4pt]
& F_{5}(x,y,z)=-\frac{\ta^*(\xi)}{2n}\bigl\{g(x,\f
y)\eta(z)+g(x,\f
z)\eta(y)\bigr\};\\[4pt]
&F_{6}(x,y,z)=\frac{\ta(\xi)}{2n}\bigl\{g(\f x,\f y)\eta(z)+g(\f
x,\f z)\eta(y)\bigr\}+\frac{\ta^*(\xi)}{2n}\bigl\{g(x,\f
y)\eta(z)+g(x,\f
z)\eta(y)\bigr\}\\[4pt]
\end{split}
\end{equation}
\begin{equation}
\begin{split}
&\phantom{F_{6}(x,y,z)=}
+\frac{1}{4}%
   \bigl[F(\f^2 x,\f^2 y,\xi)+F(\f^2 y,\f^2 x,\xi)-F(\f x,\f y,\xi)-F(\f y,\f x,\xi)\bigr]\eta(z)\\[4pt]
&\phantom{F_{6}(x,y,z)=}%
+\frac{1}{4}\bigl[F(\f^2 x,\f^2 z,\xi)+F(\f^2 z,\f^2 x,\xi)-F(\f x,\f z,\xi)-F(\f z,\f x,\xi)\bigr]\eta(y);\\[4pt]
&F_{7}(x,y,z)=\frac{1}{4}
    \bigl[F(\f^2 x,\f^2 y,\xi)-F(\f^2 y,\f^2 x,\xi)-F(\f x,\f y,\xi)+F(\f y,\f x,\xi)\bigr]\eta(z)\\[4pt]
&\phantom{F_{7}(x,y,z)=} %
+\frac{1}{4}
    \bigl[F(\f^2 x,\f^2 z,\xi)-F(\f^2 z,\f^2 x,\xi)-F(\f x,\f z,\xi)+F(\f z,\f x,\xi)\bigr]\eta(y);\\[4pt]
&F_{8}(x,y,z)=\frac{1}{4}
    \bigl[F(\f^2 x,\f^2 y,\xi)+ F(\f^2 y,\f^2 x,\xi)+F(\f x,\f y,\xi)+ F(\f y,\f x,\xi)\bigr]\eta(z)\\[4pt]
&\phantom{F_{8}(x,y,z)=}+\frac{1}{4}
    \bigl[F(\f^2 x,\f^2 z,\xi)+ F(\f^2 z,\f^2 x,\xi)+F(\f x,\f z,\xi)+ F(\f z,\f x,\xi)\bigr]\eta(y);\\[4pt]
&F_{9}(x,y,z)=\frac{1}{4}
    \bigl[F(\f^2 x,\f^2 y,\xi)-F(\f^2 y,\f^2 x,\xi)+F(\f x,\f y,\xi)-F(\f y,\f x,\xi)\bigr]\eta(z)\\[4pt]
&\phantom{F_{9}(x,y,z)=} +\frac{1}{4}
    \bigl[F(\f^2 x,\f^2 z,\xi)-F(\f^2 z,\f^2 x,\xi)+F(\f x,\f z,\xi)-F(\f z,\f x,\xi)\bigr]\eta(y).
\end{split}
\end{equation}
\end{subequations}
\end{proposition}

\begin{proof}
\lemref{lem-L10} implies that the tensor
$\frac{1}{2}\left\{F+L_{1}(F)\right\}$ is the projection of
$F\in\W_2$ in $\W_2^+=\W_{2,1}\oplus\W_{2,3}$ and moreover
$\frac{1}{2}\left\{F-L_{2}(F)\right\}$ is the projection of
$F\in\W_2^+$ in $\W_{2,1}$. %
Thus, we find the expression of the projection $p_{2,1}$ of $F$
from $\W_2$ to $\W_{2,1}$ in terms of $L_{1}$ and $L_{2}$, namely
%
\[
p_{2,1}(F)=\frac{1}{4}\left\{F-L_{1}(F)+L_{2}(F)-L_{2}\circ
L_{1}(F)\right\}.
\]
In a similar way we treat with the projections $p_{2,k}(F)$ in $\W_{2,k}$ $(k=2,3,4)$.
After that, using \eqref{L12}, we get the following expressions
\begin{equation}\label{p2k}
\begin{split}
&p_{2,1}(F)(x,y,z)=\frac{1}{4}
                    \bigl[F(\f^2 x,\f^2 y,\xi)+ F(\f^2 y,\f^2 x,\xi)
-F(\f x,\f y,\xi)- F(\f y,\f x,\xi)\bigr]\eta(z)\\[4pt]
&\phantom{p_{2,1}(F)(x,y,z)=}+\frac{1}{4}
                    \bigl[F(\f^2 x,\f^2 z,\xi)+ F(\f^2 z,\f^2 x,\xi)
-F(\f x,\f z,\xi)- F(\f z,\f x,\xi)\bigr]\eta(y);\\[4pt]
&p_{2,2}(F)(x,y,z)=\frac{1}{4}
                    \bigl[F(\f^2 x,\f^2 y,\xi)- F(\f^2 y,\f^2 x,\xi)
-F(\f x,\f y,\xi)+ F(\f y,\f x,\xi)\bigr]\eta(z)\\[4pt]
&\phantom{p_{2,2}(F)(x,y,z)=}+\frac{1}{4}
                    \bigl[F(\f^2 x,\f^2 z,\xi)- F(\f^2 z,\f^2 x,\xi)
-F(\f x,\f z,\xi)+ F(\f z,\f x,\xi)\bigr]\eta(y);\\[4pt]
&p_{2,3}(F)(x,y,z)=\frac{1}{4}
                    \bigl[F(\f^2 x,\f^2 y,\xi)+ F(\f^2 y,\f^2 x,\xi)
+F(\f x,\f y,\xi)+ F(\f y,\f
x,\xi)\bigr]\eta(z)\\[4pt]
&\phantom{p_{2,3}(F)(x,y,z)=}+\frac{1}{4}
                    \bigl[F(\f^2 x,\f^2 z,\xi)+ F(\f^2 z,\f^2 x,\xi)
+F(\f x,\f z,\xi)+ F(\f z,\f
x,\xi)\bigr]\eta(y);\\[4pt]
&p_{2,4}(F)(x,y,z)=\frac{1}{4}
                    \bigl[F(\f^2 x,\f^2 y,\xi)- F(\f^2 y,\f^2 x,\xi)
+F(\f x,\f y,\xi)- F(\f y,\f x,\xi)\bigr]\eta(z)\\[4pt]
&\phantom{p_{2,4}(F)(x,y,z)=}+\frac{1}{4}
                    \bigl[F(\f^2 x,\f^2 z,\xi)- F(\f^2 z,\f^2 x,\xi)
+F(\f x,\f z,\xi)- F(\f z,\f x,\xi)\bigr]\eta(y).
\end{split}
\end{equation}

By virtue of \coref{cor-t-Wi}, \eqref{t}  and \eqref{p2k}, we
establish that the Lee forms $\ta$ and $\ta^*$ of $F$ are zero in
$\W_{2,k}$ $(k=2,3,4)$. We have no additional conditions for $\ta$
and $\ta^*$ in $\W_{2,1}$. Then, $\W_{2,1}$ can be additionally
decomposed to three subspaces determined by the conditions
$\ta=0$, $\ta^*=0$ and $\ta=\ta^*=0$, respectively, \ie
\begin{equation*}\label{W21}
\W_{2,1}=\W_{2,1,1}\oplus \W_{2,1,2}\oplus \W_{2,1,3},
\end{equation*}
where
\begin{equation*}\label{W21s}
\begin{array}{c}
\W_{2,1,1}=\left\{F\in \W_{2,1}\ \vert\
\ta^*=0\right\},\qquad
\W_{2,1,2}=\left\{F\in \W_{2,1}\ \vert\ \ta=0\right\}, \\[4pt]
\W_{2,1,3}=\left\{F\in \W_{2,1}\ \vert\ \ta=0,\ \ta^*=0\right\}.
\end{array}
\end{equation*}

According to \eqref{Fi}, the classes $\F_4$ and $\F_5$ are defined
by explicit expressions of $F$, which have the form of $p_{2,1}(F)$ in \eqref{p2k}.
Hence we conclude that $\F_4$ and $\F_5$ are the subspaces $\W_{2,1,1}$ and $\W_{2,1,2}$
of $\W_{2,1}$, respectively, and
the projections of $F$ on them 
have the form given in the first line of  \eqref{F4-9}. Therefore,
the equality for $F_6$ in \eqref{F4-9} follows from the form of
$p_{2,1}(F)$ in \eqref{p2k} and the fact that $\F_6$ coincides
with $\W_{2,1,3}$.

The form of $p_{2,2}(F)$, $p_{2,3}(F)$ and $p_{2,4}(F)$ in \eqref{p2k} satisfies
the conditions in \eqref{Fi} for the subspace $\F_7$, $\F_8$ and $\F_9$, respectively.

Thus, the subspaces $\W_{2,1,1}$, $\W_{2,1,2}$, $\W_{2,1,3}$, $\W_{2,2}$, $\W_{2,3}$, $\W_{2,4}$
correspond to
the classes
$\F_4$, $\F_5$, $\F_6$, $\F_7$, $\F_8$, $\F_9$,
respectively.
\end{proof}

Therefore, the component of $F$ on $\W_2$ is
\begin{equation}\label{p2F}
p_2(F)=F_4+F_5+F_6+F_7+F_8+F_9.
\end{equation}

\subsection{The decomposition of $\W_3$ and $\W_4$}

Finally, since $\F_{10}$ and $\F_{11}$ are determined in
\eqref{Fi} by an expression of $F$ which coincide with the
conditions in \eqref{WF} for $\W_3$ and $\W_4$, respectively, we
have the following
\begin{proposition}\label{prop-F1011}
Let $F\in\F$ and $F_{l}$ $(l=10,11)$ be the projections of $\F$ in
the subspaces $\F_{l}$. Then we have
\begin{equation}\label{F1011}
\begin{array}{l}
F_{10}(x,y,z)=\eta (x) F(\xi, \f^2 y, \f^2 z);\\[4pt]
F_{11}(x,y,z)= -\eta (x) \left\{\eta(y) F(\xi, \xi,\f^2 z) + \eta(z)
F(\xi, \f^2 y, \xi)\right\}.
\end{array}
\end{equation}
\end{proposition}

Therefore, the components of $F$ on $\W_3$ and $\W_4$ are
\begin{equation}\label{p34F}
p_3(F)=F_{10},\qquad p_4(F)=F_{11},
\end{equation}
respectively. Then, bearing in mind \eqref{p1F}, \eqref{p2F} and
\eqref{p34F}, we obtain that
\[
F(x,y,z)=\sum_{i=1}^{11}F_i(x,y,z),
 \]
where the components $F_i$ of $F$ in the corresponding subspaces
$\F_i$ $(i=1,\dots,11)$ of $\F$ are determined in \propref{prop-F123},
\propref{prop-F4-9} and \propref{prop-F1011}.

In conclusion we give
\begin{theorem}\label{thm-Fi}
The almost contact B-metric manifold $\M$ belongs to the basic
class $\F_i$ $(i=1,\dots,11)$ if and only if the structure tensor
$F$ satisfies the condition $F=F_i$, where the components $F_i$ of
$F$ are given in \eqref{F123}, \eqref{F4-9} and \eqref{F1011}.
\end{theorem}

It is easy to conclude that an almost contact B-metric manifold
belongs to a direct sum of two or more basic classes, \ie
$\M\in\F_i\oplus\F_j\oplus\cdots$, if and only if the structure
tensor $F$ on $\M$ is the sum of the corresponding components
$F_i$, $F_j$, $\ldots$ of $F$, \ie the following condition is
satisfied $F=F_i+F_j+\cdots$.

\section{The components of the structure tensor for dimension 3}\label{sec-com3}

In this section we are interesting in the lowest dimension of the manifolds
under consideration, \ie we consider the case of $\dim{M}=3$ for $\M$.

Let us denote the components $F_{ijk}=F(e_i,e_j,e_k)$ of the structure tensor
$F$ with respect to a $\f$-basis $\left\{e_i\right\}_{i=0}^2=\left\{e_0=\xi,e_1=e,e_2=\f e\right\}$,
which satisfies the following conditions
\begin{equation}\label{gij}
g(e_0,e_0)=g(e_1,e_1)=-g(e_2,e_2)=1,\quad g(e_i,e_j)=0,\; i\neq j \in \{0,1,2\}.
\end{equation}

Then, using \eqref{t} and \eqref{gij}, we obtain the components of the Lee forms
with respect to the basis $\left\{e_i\right\}_{i=0}^2$ as follows
\begin{equation}\label{t3}
\begin{array}{lll}
\ta_0=F_{110}-F_{220},\qquad & \ta_1=F_{111}-F_{221},\qquad &\ta_2=F_{112}-F_{211},\\[4pt]
\ta^*_0=F_{120}+F_{210},\qquad &\ta^*_1=F_{112}+F_{211},\qquad &\ta^*_2=F_{111}+F_{221},\\[4pt]
\om_0=0,\qquad &\om_1=F_{001},\qquad &\om_2=F_{002}.
\end{array}
\end{equation}


Let us consider arbitrary vectors $x,y,z\in T_pM$. Therefore we have
$x=x^ie_i$, $y=y^ie_i$, $z=z^ie_i$ with respect to $\left\{e_i\right\}_{i=0}^2$.

By direct computations we obtain
\begin{proposition}\label{prop-Fi}
The components $F_i$ $(i=1,2,\dots,11)$ of the structure tensor
$F$ in the corresponding basic classes $\F_i$ are the following
\begin{equation}\label{Fi3}
\begin{array}{l}
F_{1}(x,y,z)=\left(x^1\ta_1-x^2\ta_2\right)\left(y^1z^1+y^2z^2\right),
\qquad \ta_1=F_{111}=F_{122},\quad \ta_2=-F_{211}=-F_{222}; \\[4pt]
F_{2}(x,y,z)=F_{3}(x,y,z)=0;
\\[4pt]
F_{4}(x,y,z)=\frac{1}{2}\ta_0\Bigl\{x^1\left(y^0z^1+y^1z^0\right)
-x^2\left(y^0z^2+y^2z^0\right)\bigr\},
%
\quad \frac{1}{2}\ta_0=F_{101}=F_{110}=-F_{202}=-F_{220};\\[4pt]
F_{5}(x,y,z)=\frac{1}{2}\ta^*_0\bigl\{x^1\left(y^0z^2+y^2z^0\right)
+x^2\left(y^0z^1+y^1z^0\right)\bigr\},
%
\quad \frac{1}{2}\ta^*_0=F_{102}=F_{120}=F_{201}=F_{210};\\[4pt]
F_{6}(x,y,z)=F_{7}(x,y,z)=0;\\[4pt]
F_{8}(x,y,z)=\lm\bigl\{x^1\left(y^0z^1+y^1z^0\right)
+x^2\left(y^0z^2+y^2z^0\right)\bigr\},
\quad \lm=F_{101}=F_{110}=F_{202}=F_{220};\\[4pt]
F_{9}(x,y,z)=\mu\bigl\{x^1\left(y^0z^2+y^2z^0\right)
-x^2\left(y^0z^1+y^1z^0\right)\bigr\},
\quad \mu=F_{102}=F_{120}=-F_{201}=-F_{210};\\[4pt]
F_{10}(x,y,z)=\nu x^0\left(y^1z^1+y^2z^2\right),\quad
\nu=F_{011}=F_{022};\\[4pt]
F_{11}(x,y,z)=x^0\bigl\{\left(y^1z^0+y^0z^1\right)\om_{1}
+\left(y^2z^0+y^0z^2\right)\om_{2}\bigr\},
\quad \om_1=F_{010}=F_{001},\quad \om_2=F_{020}=F_{002}.
\end{array}
\end{equation}
\end{proposition}
\begin{proof}
Using \thmref{thm-Fi} and the expressions of the components $F_i$ of $F$ for
the corresponding classes $\F_i$ $(i=1,\dots,11)$, determined by
\eqref{F123}, \eqref{F4-9} and \eqref{F1011}, the equalities \eqref{gij},
\eqref{t3} and the properties \eqref{F-prop} of $F$, we obtain the
corresponding explicit expression of $F_i$  for dimension 3.
\end{proof}

According to \thmref{thm-Fi} and \propref{prop-Fi}, we obtain
\begin{theorem}\label{thm-3D}
The class of almost contact B-metric manifolds of dimension 3
is
\[
\F_1 \oplus \F_4 \oplus \F_5 \oplus \F_8 \oplus \F_9 \oplus
\F_{10} \oplus \F_{11},
\]
\ie the basic classes $\F_2$, $\F_3$,
$\F_6$, $\F_7$ are restricted to the special class $\F_0$.
\end{theorem}


\section{Some examples}\label{sec-exm}

\subsection{Time-like sphere as a manifold from the class
$\F_4\oplus\F_{5}$}

In \cite{GaMiGr}, an example of an almost contact manifold with
B-metric is given. It is constructed as a time-like sphere of
$\R^{2n+2}$ with complex structure and Norden metric. Namely, let
$\R^{2n+2}=\Bigl\{\left(u^1,\dots,u^{n+1};v^1,\dots,v^{n+1}\right)\
\big|\ \allowbreak{}u^i,v^i\in\R\Bigr\}$ be considered as a
complex Riemannian manifold with the canonical complex structure
$J$ and the metric $g$ defined by
\[
J\frac{\partial}{\partial u^i}=\frac{\partial}{\partial v^i},\quad
J\frac{\partial}{\partial v^i}=-\frac{\partial}{\partial u^i},
\quad g(x,x)=-\delta_{ij}\lm^i\lm^j+\delta_{ij}\mu^i\mu^j,
\]
where $x=\lm\frac{\partial}{\partial
u^i}+\mu\frac{\partial}{\partial v^i}$. Identifying the point $p$
in $\R^{2n+2}$ with its position vector $Z$, we define the unit
time-like sphere
\[
S^{2n+1}:\; g(Z,Z)=-1.
\]
The almost contact structure is determined by
\[
\xi=\sin t \cdot Z+\cos t\cdot JZ,\quad
 Jx=\f
x+\eta(x)J\xi,
\]
where
$t=\arctan\bigl(g(z,Jz)\bigr)\in\left(-\frac{\pi}{2};\frac{\pi}{2}\right)$
and $x$, $\f x$ $\in T_pS^{2n+1}$. The metric on the hypersurface
is the restriction of $g$ and it is denote by the same letter.
Then $\left(S^{2n+1},\f,\xi,\eta,g\right)$ is an almost contact
B-metric manifold. It belongs to the class $\F_4\oplus\F_5$
because the structure tensor has the following form
\[
F(x,y,z)=-\cos t\ \{g(\f x,\f y)\eta(z)+g(\f x,\f z)\eta(y)\}
-\sin t\ \{g(x,\f y)\eta(z)+g(x,\f z)\eta(y)\},
\]
where $\cos t=\frac{\ta(\xi)}{2n}$, $\sin t=\frac{\ta^*(\xi)}{2n}$
and then we obtain the following expression of $F$, bearing in
mind \eqref{F4-9}:
\[
F=F_4+F_5.
\]

If we consider the 3-dimensional unit time-like sphere
$(S^3,\f,\xi,\eta,g)$ then we have the following form of the
structure tensor, using \eqref{Fi3}:
\[
F(x,y,z)=\frac{1}{2}\Bigl\{\left(\ta_0x^1+\ta^*_0x^2\right)\left(y^0z^1+y^1z^0\right)
+\left(\ta^*_0
x^1-\ta_0x^2\right)\left(y^0z^2+y^2z^0\right)\Bigr\}
\]
where $\frac{1}{2}\ta_0=F_{101}=F_{110}=-F_{202}=-F_{220}$ and
$\frac{1}{2}\ta^*_0=F_{102}=F_{120}=F_{201}=F_{210}$ with respect
to the orthonormal $\f$-basis $\left\{e_i\right\}_{i=0}^2$.

\subsection{Lie group as a manifold from the class
$\F_9\oplus\F_{10}$}

Let $L$ be a $(2n+1)$-dimensional real connected Lie group and its
associated Lie algebra with a global basis $\{E_{0},E_{1},\dots,
E_{2n}\}$ of left invariant vector fields on $L$ defined by
\begin{equation}\label{com}
    [E_0,E_i]=-a_iE_i-a_{n+i}E_{n+i},\quad
    [E_0,E_{n+i}]=-a_{n+i}E_i+a_{i}E_{n+i},
\end{equation}
where $a_1,\dots,a_{2n}$ are real constants and $[E_j,E_k]=0$ in
other cases.
Let an invariant almost contact structure be defined for
$i\in\{1,\dots,n\}$ by
\begin{equation}\label{strL}
\f E_0=0,\quad \f E_i=E_{n+i},\quad \f E_{n+i}=- E_i,\quad \xi=
E_0,\quad \eta(E_0)=1,\quad \eta(E_i)=\eta(E_{n+i})=0.
\end{equation}
Let $g$ be a pseudo-Riemannian metric such that for
$i\in\{1,\dots,n\}$ and $j, k \in\{1,\dots,2n\}$, $j\neq k$ the
following equalities are valid
\begin{equation}\label{gL}
  g(E_0,E_0)=g(E_i,E_i)=-g(E_{n+i},E_{n+i})=1, \quad
  g(E_0,E_j)=g(E_j,E_k)=0.
\end{equation}
Thus, because of \eqref{str}, the induced $(2n+1)$-dimensional
manifold $(L,\f, \xi, \eta, g)$ is an almost contact B-metric
manifold.

Let us remark that in \cite{Ol} the same Lie group with the same
almost contact structure but equipped with a compatible Riemannian
metric is studied as an almost cosymplectic manifold.

Let us consider the constructed almost contact B-metric manifold
$(L,\f, \xi, \eta, g)$ in dimension 3, \ie for $n=1$.

By virtue of \eqref{com} and \eqref{gL} for $n=1$,  and using the
Koszul equality
\begin{equation}\label{Kosz}
2g\left(\n_{E_i}E_j,E_k\right)
=g\left([E_i,E_j],E_k\right)+g\left([E_k,E_i],E_j\right)
+g\left([E_k,E_j],E_i\right)
\end{equation}
for the Levi-Civita connection $\n$ of $g$, we obtain
\begin{equation}\label{nEi}
    \n_{E_1}E_1=\n_{E_2}E_2=-a_1E_0,\quad
    \n_{E_0}E_1=-a_2E_2,\quad \n_{E_0}E_2=-a_2E_1,\quad
    \n_{E_1}E_0=a_1E_1,\quad \n_{E_2}E_0=-a_1E_2.
\end{equation}

Then, using the latter equalities, \eqref{strL} and \eqref{F=nfi},
we get the following nonzero components $F_{ijk}=F(E_i,E_j,E_k)$
of the structure tensor:
\[
F_{011}=F_{022}=-2a_2,\quad F_{102}=F_{120}=-F_{201}=-F_{210}=a_1.
\]
Thus, we establish the following form of $F$ for arbitrary vectors
$x=x^iE_i$, $y=y^iE_i$, $z=z^iE_i$
\[
F(x,y,z)=-2a_2x^0\left(y^1z^1+y^2z^2\right)
+a_1\left\{z^0\left(x^1y^2-x^2y^1\right)+y^0\left(x^1z^2-x^2z^1\right)\right\}.
\]
The latter equality implies that $F$ is represented in the form
\[
F(x,y,z)=F_9(x,y,z)+F_{10}(x,y,z),
\]
bearing in mind \eqref{Fi3} for $\mu=a_1$, $\nu=-2a_2$; or
alternatively, the corresponding equalities from \eqref{F4-9} and
\eqref{F1011}. Therefore, we prove that the constructed
3-dimensional manifold belongs to the class $\F_9\oplus\F_{10}$.


\end{document}